\documentclass[10pt]{amsart}
\usepackage{latexsym,amsfonts,amssymb}

\newtheorem{theorem}{Theorem}[section]
\newtheorem{lemma}[theorem]{Lemma}
\newtheorem{corollary}[theorem]{Corollary}
\newtheorem{question}[theorem]{Question}
\theoremstyle{definition}

\theoremstyle{remark}

\begin{document}

\title{Local properties on the remainders of the topological groups}

\author{Fucai Lin}%
\address{Fucai Lin: Department of Mathematics and Information Science,
Zhangzhou Normal University, Zhangzhou 363000, P. R. China.}%
\email{linfucai2008@yahoo.com.cn}


\thanks{Supported by the NSFC (No. 10971185, No. 10971186) and the Educational Department of Fujian Province (No. JA09166) of China.}

\begin{abstract} When does a topological group $G$ have a
Hausdorff compactification $bG$ with a remainder belonging to a
given class of spaces? In this paper, we mainly improve some results
of A.V. Arhangel'ski\v{\i} and C. Liu's. Let $G$ be a non-locally
compact topological group and $bG$ be a compactification of $G$. The
following facts are established: (1) If $bG\setminus G$ has a locally a point-countable $p$-metabase and
$\pi$-character of $bG\setminus G$ is countable, then $G$ and $bG$
are separable and metrizable; (2) If $bG\setminus G$  has locally a
$\delta\theta$-base, then $G$ and $bG$ are separable and metrizable;
(3) If $bG\setminus G$ has locally a quasi-$G_{\delta}$-diagonal,
then $G$ and $bG$ are separable and metrizable. Finally, we give a
partial answer for a question, which was posed by C. Liu in
\cite{LC}.
\end{abstract}

\maketitle

\parskip 0.15cm

\section{Introduction}
By a remainder of a space $X$ we understand the subspace
$bX\setminus X$ of a Hausdorff compactification $bX$ of $X$. In
\cite{A, A1, A2, HM, LC}, many topologists studied the following
question of a Hausdorff compactification: When does a Tychonoff
space $X$ have a Hausdorff compactification $bX$ with a remainder
belonging to a given class of spaces? A famous classical result in
this study is the following theorem of M. Henriksen and J.
Isbell \cite{HM}:

\medskip
{\bf (M. Henriksen and J. Isbell)} A space $X$ is of countable type
if and only if the remiander in any (in some) compactification of
$X$ is Lindel\"{o}f.

Recall that a space $X$ is of {\it countable type} \cite{E} if every
compact subspace $F$ of $X$ is contained in a compact subspace
$K\subset X$ with a countable base of open neighborhoods in $X$.
Suppose that $X$ is a non-locally compact topological group, and
that $bX$ is a compactification of $X$. In \cite{A1}, A.V.
Arhangel'ski\v{\i} showed that if the remainder $Y=bX\setminus X$
has a $G_{\delta}$-diagonal or a point-countable base, then both $X$
and $Y$ are separable and metrizable. In \cite{LC}, C. Liu improved
the results of A.V. Arhangel'ski\v{\i}, and proved that if $Y$
satisfies one of the following conditions (i) and (ii), then $X$ and
$bX$ are separable and metrizable.

(i) $Y=bX\setminus X$ is a quotient $s$-image of a metrizable space,
and $\pi$-character of $Y$ is countable;

(ii) $Y=bX\setminus X$ has locally a $G_{\delta}$-diagonal.

In this paper, we mainly concerned with the following statement, and
under what condition $\Phi$ it is true.

\medskip
{\bf Statement} Suppose that $G$ is a non-locally compact
topological group, and that $Y=bG\setminus G$ has locally a
property-$\Phi$. Then $G$ and $bG$ are separable and metrizable.

\medskip
Recall that a space $X$ has {\it locally a property-$\Phi$} if for
each point $x\in X$ there exists an open set $U$ with $x\in U$ such
that $U$ has a property-$\Phi$.

\medskip
In Section 2 we mainly study some local properties on the remainders
of the topological group $G$ such that $G$ and $bG$ are separable
and metrizable if the $\pi$-character of $bG\setminus G$ is
countable. Therefore, we extend some results of A.V.
Arhangel'ski\v{\i} and C. Liu.

In Section 3 we prove that if the remainders of a topological group
$G$ has locally a quasi-$G_{\delta}$-diagonal, then $G$ and $bG$ are
separable and metrizable. Therefore, we improve a result of C. Liu
in \cite{LC}. Also, we study the remainders that are the unions of
$G_{\delta}$-diagonals.

In Section 4 we mainly give a partial answer for a question, which
was posed by C. Liu in \cite{LC}. Finally, we also study the
remainders that are locally hereditarily $D$-spaces.

Recall that a family $\mathcal{U}$ of non-empty open sets of a space
$X$ is called a {\it $\pi$-base} if for each non-empty open set $V$
of $X$, there exists an $U\in\mathcal{U}$ such that $U\subset V$. The
{\it $\pi$-character} of $x$ in $X$ is defined by $\pi\chi(x,
X)=\mbox{min}\{|\mathcal{U}|:\mathcal{U}\ \mbox{is a local}\
\pi\mbox{-base at}\ x\ \mbox{in}\ X\}$. The {\it $\pi$-character of
$X$} is defined by $\pi\chi(X)=\mbox{sup}\{\pi\chi(x, X):x\in
X\}$.

The $p$-spaces are a class of generalized metric spaces \cite{A3}. It is
well-known that every metrizable space is a $p$-space, and every
$p$-space is of countable type.

Throughout this paper, all spaces are assumed to be Hausdorff.
The positively natural numbers is denoted by $\mathbb{N}$. We refer the
readers to \cite{E, Gr} for notations and terminology not explicitly
given here.

\section{Remainders with the countable $\pi$-characters}
Let $\mathcal{A}$ be a collection of subsets of $X$. $\mathcal{A}$
is a {\it $p$-network} \cite{BD} for $X$ if for distinct points $x,
y\in X$, there exists an $A\in \mathcal{A}$ such that $x\in A\subset
X-\{y\}$. The collection $\mathcal{A}$ is called a {\it $p$-base} (i.e., {\it
$T_{1}$-point-separating open cover}) \cite{BD} for $X$ if
$\mathcal{A}$ is a $p$-network and each element of $\mathcal{A}$ is
an open subset of $X$. The collection $\mathcal{A}$ is a {\it
$p$-metabase} \cite{Ls} (in \cite{BD}, $p$-metabase is denoted by the
condition (1.5)) for $X$ if for distinct points $x, y\in X$, there
exists an $\mathcal{F}\in \mathcal{A}^{<\omega}$ such that $x\in
(\cup\mathcal{F})^{\circ}\subset\cup\mathcal{F}\subset X-\{y\}$.

First, we give some technique lemmas.

\begin{lemma}\cite{A} \label{l0}
If $X$ is a Lindel$\ddot{o}$f $p$-space, then any remainder of $X$
is a Lindel$\ddot{o}$f $p$-space.
\end{lemma}

\begin{lemma}\cite{LC} \label{l6}
Let $G$ be a non-locally compact topological group. Then $G$ is
locally separable and metrizable if for each point $y\in
Y=bG\setminus G$, there is an open neighborhood $U(y)$ of $y$ such
that every countably compact subset of $U(y)$ is metrizable and
$\pi$-character of $Y$ is countable .
\end{lemma}

\begin{lemma} \label{l1}
Suppose that $X$ has a point-countable $p$-metabase. Then each
countably compact subset of $X$ is a compact, metrizable,
$G_{\delta}$-subset\footnote{A subset $K$ of $X$ is called a {\it
$G_{\delta}$-subset} of $X$ if $K$ is the intersection of countably
open subsets of $X$.} of $X$.
\end{lemma}

\begin{proof}
Suppose that $\mathcal{U}$ is a point-countable $p$-metabase of $X$,
 and that $K$ is a countably compact subset of $X$. Then $K$ is compact by
\cite{BD}. According to a generalized Mi$\breve{s}\breve{c}$enko's
Lemma in \cite[Lemma 6]{YL}, there are only countably many minimal
neighborhood-covers\footnote{Let $\mathcal{P}$ be a collection of
subsets of $X$ and $A\subset X$. The collection $\mathcal{P}$ is a {\it
neighborhood-cover} of $A$ if $A\subset (\cup\mathcal{P})^{\circ}$.
A neighborhood-cover $\mathcal{P}$ of $A$ is a {\it minimal
neighborhood-cover} if for each $P\in \mathcal{P}$,
$\mathcal{P}\setminus\{P\}$ is not a neighborhood-cover of $A$.} of
$K$ by finite elements of $\mathcal{U}$, say $\{\mathcal{V}(n):n\in
\mathbb{N}\}$. Let $V(n)=\cup\mathcal{V}(n)$. Then
$K\subset\cap\{V(n):n\in \mathbb{N}\}$. Suppose that $x\in
X\setminus K$. For each point $y\in K$, there is an
$\mathcal{F}_{y}\in\mathcal{U}^{<\omega}$ with $y\in
(\cup\mathcal{F}_{y})^{\circ}\subset \cup\mathcal{F}_{y}\subset
X-\{x\}$. Then there is some sub-collection of
$\cup\{\mathcal{F}_{y}:y\in K\}$ is a minimal finite
neighborhood-covers of $K$ since $K$ is compact. Therefore, we
obtain one of the collections $\mathcal{V}(n)$ with $K\subset
V(n)=\cup\mathcal{V}(n)\subset X-\{x\}$.
\end{proof}

\begin{lemma} \label{l2}
Suppose that $X$ is a Lindel$\ddot{o}$f space with locally a
point-countable $p$-metabase. Then $X$ has a point-countable
$p$-metabase.
\end{lemma}

\begin{proof}
For each point $x\in X$, there is an open neighborhood $U(x)$ with
$x\in U(x)$ such that $U(x)$ has a point-countable $p$-metabase
$\mathcal{F}_{x}$. Let $\mathcal{U}=\{U(x): x\in X\}$. Since $X$ is
Lindel\"{o}f, it follows that there exists a countable subfamily
$\mathcal{U}^{\prime}\subset \mathcal{U}$ such that
$X=\cup\mathcal{U}^{\prime}$. Denoted $\mathcal{U}^{\prime}$ by
$\{U_{x_{i}}: i\in\mathbb{N}\}$. Obviously,
$\mathcal{F}=\bigcup_{i}\mathcal{F}_{x_{i}}$ is a point-countable
$p$-metabase for $X$.
\end{proof}

\begin{theorem} \label{t0}
Suppose that $G$ is a non-locally compact topological group, and
that $Y=bG\setminus G$ has locally a point-countable $p$-metabase.
Then $G$ and $bG$ are separable and metrizable if $\pi$-character of
$Y$ is countable .
\end{theorem}

\begin{proof}
It is easy to see that $G$ is locally separable and metrizable by
Lemmas~\ref{l6} and \ref{l1}. Then $G$ is a $p$-space.  Hence $Y$ is
Lindel\"{o}f by Henriksen and Isbell's theorem. From Lemma~\ref{l2}
it follows  that $Y=bG\setminus G$ has a point-countable
$p$-metabase.

Claim: The space $Y$ has a $G_{\delta}$-diagonal.

Put $G=\oplus_{\alpha\in \wedge}G_{\alpha}$, where $G_{\alpha}$ is a
separable and metrizable subset for each $\alpha\in \wedge$. Let
$\zeta =\{G_{\alpha}:\alpha\in \wedge\}$, and let $F$ be the set of
all points of $bG$ at which $\zeta$ is not locally finite. Since
$\zeta$ is discrete in $G$, it follows that $F\subset bG\setminus
G$. It is easy to see that $F$ is compact. Therefore, it follows
from Lemma~\ref{l1} that $F$ is separable and metrizable. Hence $F$
has a countable network.

Let $M=Y\setminus F$. For each point $y\in M$, there is an open
neighborhood $O_{y}$ in $bG$ such that $\overline{O_{y}}\cap
F=\emptyset$. Since $\zeta$ is discrete, $\overline{O_{y}}$ meets at
most finitely many $G_{\alpha}$. Let
$L=\cup\{G_{\alpha}:G_{\alpha}\cap\overline{O_{y}}\neq\emptyset\}$.
Then $L$ is separable and metrizable. By Lemma~\ref{l0},
$\overline{L}\setminus L$ is a Lindel\"{o}f $p$-space. Obviously,
$\overline{L}\setminus L\subset Y$. Therefore,
$\overline{L}\setminus L$ has a point-countable $p$-metabase. Hence
$\overline{L}\setminus L$ is separable and metrizable by
\cite{GMT1}, which implies that $\overline{L}$ has a countable
network. It follows that $\overline{L}$ is separable and metrizable.
Clearly, $O_{y}\subset \overline{L}$ and $O_{y}\cap M$ is separable
and metrizable. Therefore, $M$ is locally separable and metrizable.
From Lemma~\ref{l1} it follows that each compact subset of $Y$ is a
$G_{\delta}$-subset of $Y$. Since $F$ is compact and $Y$ is
Lindel\"{o}f, it follows that $M$ is Lindel\"{o}f. Therefore, $M$ is
separable. Then $M$ has a countable network. So $Y$ has a countable
network, which implies that $Y$ has a $G_{\delta}$-diagonal. Thus,
Claim is verified.

Therefore, $G$ and $bG$ are separable and metrizable by
\cite[Theorem 5]{A1}.
\end{proof}

\begin{corollary}
Suppose that $G$ is a non-locally compact topological group, and
that $Y=bG\setminus G$ has locally a point-countable $p$-base. Then
$G$ and $bG$ are separable and metrizable if $\pi$-character of $Y$
is countable.
\end{corollary}

\begin{corollary}\cite{A1}\label{c1}
Suppose that $G$ is a non-locally compact topological group. If
$Y=bG\setminus G$ has a point-countable base, then $G$ and $bG$ are
separable and metrizable.
\end{corollary}

Next, we consider the remainders with locally a
$\delta\theta$-base\footnote{Recall that a collection
$\mathcal{B}=\cup_{n}\mathcal{B}_{n}$ of open subsets of a space $X$
is a {\it $\delta\theta$-base} \cite{Gr} if whenever $x\in U$ with
$U$ open, there exist an $n\in \mathbb{N}$ and a $B\in \mathcal{B}$
such that

(i) $1\leq\mbox{ord}(x, \mathcal{B}_{n})\leq\omega$;

(ii) $x\in B\subset U$.} of the topological groups.

\begin{lemma}\label{l9}
Let $X$ be a Lindel\"{o}f space with locally a $\delta\theta$-base.
Then $X$ has a $\delta\theta$-base.
\end{lemma}

\begin{proof}
For each point $x\in X$, there is an open neighborhood $U(x)$ with
$x\in U(x)$ such that $U(x)$ has a $\delta\theta$-base
$\mathcal{B}_{x}=\bigcup_{n}\mathcal{B}_{n, x}$. Let
$\mathcal{U}=\{U(x): x\in X\}$. Since $X$ is Lindel\"{o}f, it
follows that there exists a countable subfamily
$\mathcal{U}^{\prime}\subset \mathcal{U}$ such that
$X=\cup\mathcal{U}^{\prime}$. Denoted $\mathcal{U}^{\prime}$ by
$\{U_{x_{i}}: i\in\mathbb{N}\}$. Obviously, $\mathcal{B}=\bigcup_{i,
n}\mathcal{B}_{n, x_{i}}$ is a $\delta\theta$-base for $X$.
\end{proof}

\begin{theorem}
Let $G$ be a non-locally compact topological group. If
$Y=bG\setminus G$ has locally a $\delta\theta$-base. Then $G$ and
$bG$ are separable and metrizable.
\end{theorem}

\begin{proof}
Obviously, $Y$ is first countable. By \cite[Propostion 2.1]{BH},
each countably compact subset of $Y$ is a compact, metrizable,
$G_{\delta}$-subset of $Y$. From Lemma~\ref{l6} it follows that $G$
is locally separable and metrizable. Then $G$ is a $p$-space. Hence
$Y$ is Lindel\"{o}f by Henriksen and Isbell's theorem. From
Lemma~\ref{l9} it follows that $Y=bG\setminus G$ has a
$\delta\theta$-base.

By the same notations in Theorem~\ref{t0}, it is easy to see from
\cite[Propostion 2.1]{BH} that $F\subset bG\setminus G$ is compact
and metrizable in view of the proof of Theorem~\ref{t0}. By
\cite[Corollary 8.3]{Gr} and Lemma~\ref{l0}, $\overline{L}\setminus
L$ is separable and metrizable. In view of the proof of
Theorem~\ref{t0}, $G$ and $bG$ are separable and metrizable by
\cite[Propostion 2.1]{BH}.
\end{proof}

\begin{corollary}\cite{LC}
Let $G$ be a non-locally compact topological group. If
$Y=bG\setminus G$ is locally a quasi-developable\footnote{A space
$X$ is {\it quasi-developable} if there exists a sequence
$\{\mathcal {G}_{n}\}_{n}$ of families of open subsets of $X$ such
that for each point $x\in X$, $\{\mbox{st}(x, \mathcal {G}_{n}):n\in
\mathbb{N}, \mbox{st}(x, \mathcal {G}_{n})\not=\emptyset\}$ is a
base at $x$.}. Then $G$ and $bG$ are separable and metrizable.
\end{corollary}

Finally, we consider the remainders with locally a
c-semistratifiable space of the topological group.

Let $X$ be a topological space. $X$ is called a {\it
c-semistratifiable space}(CSS) \cite{MH} if for each compact subset
$K$ of $X$ and each $n\in\mathbb N$ there is an open set $G(n, K)$
in $X$ such that:

(i) $\cap\{G(n, K):n\in \mathbb{N}\}=K$;

(ii) $G(n+1, K)\subset G(n, K)$ for each $n\in \mathbb{N}$; and

(iii) if for any compact subsets $K, L$ of $X$ with $K\subset L$,
then $G(n, K)\subset G(n, L)$ for each $n\in \mathbb{N}$.

\begin{theorem}\label{t1}
Suppose that $G$ is a non-locally compact topological group, and
that $Y=bG\setminus G$ is locally a $CSS$-space. Then $G$ and $bG$
are separable and metrizable if $\pi$-character of $Y$ is countable.
\end{theorem}

\begin{proof}
By \cite[Proposition 3.8(c)]{BH} and the definition of $CSS$-spaces,
it is easy to see that each countably compact subset of $Y$ is a
compact, metrizable, $G_{\delta}$-subset of $Y$. From Lemma~\ref{l6}
it follows that $G$ is locally separable and metrizable. Then $G$ is
a $p$-space. Hence $Y$ is Lindel\"{o}f by Henriksen and Isbell's
theorem. From Lemma~\ref{l9} it follows that $Y=bG\setminus G$ is a
$CSS$-space by \cite[Proposition 3.5]{BH}.

By the same notations in Theorem~\ref{t0}, it is easy to see from
\cite[Proposition 3.8]{BH} that $F\subset bG\setminus G$ is compact
and metrizable in view of the proof of Theorem~\ref{t0}. By
\cite[Proposition 3.8]{BH}, $\overline{L}\setminus L$ is separable
and metrizable. In view of the proof of Theorem~\ref{t0}, it is easy
to see that $G$ and $bG$ are separable and metrizable.
\end{proof}

\begin{corollary}
Suppose that $G$ is a non-locally compact topological group, and
that $Y=bG\setminus G$ is locally a
$\sigma^{\sharp}$-space\footnote{A space $X$ is called a {\it
$\sigma^{\sharp}$-space} \cite{MH} if $X$ has a
$\sigma$-closure-preserving closed $p$-network.}. Then $G$ and $bG$
are separable and metrizable if $\pi$-character of $Y$ is countable.
\end{corollary}

\begin{proof}
By \cite[Lemma 3.1]{BH}, it follows that every
$\sigma^{\sharp}$-space is a $CSS$-space. Hence $G$ and $bG$ are
separable and metrizable by Theorem ~\ref{t1}.
\end{proof}

\begin{question}
Let $G$ be a non-locally compact topological group. If
$Y=bG\setminus G$ satisfies the following conditions (1) and (2),
are $G$ and $bG$ separable and metrizable?
\begin{enumerate}
\item For each point $y\in Y$, there exists an open neighborhood $U(y)$ of
$y$ such that every countably compact subset of $U(y)$ is metirzable
and $G_{\delta}$-subset of $U(y)$;

\item $\pi$-character of $Y$ is countable .
\end{enumerate}
\end{question}

 \vskip 0.5cm
\section{Remainders that are locally quasi-$G_{\delta}$-diagonals, and that are unions}
First, we study the remainders with locally a
quasi-$G_{\delta}$-diagonal\footnote{A space $X$ has a {\it
quasi-$G_{\delta}$-diagonal} \cite{HR} if there exists a sequence
$\{\mathcal {G}_{n}\}_{n}$ of families of open subsets of $X$ such
that for each point $x\in X$, $\{\mbox{st}(x, \mathcal {G}_{n}):n\in
\mathbb{N}, \mbox{st}(x, \mathcal {G}_{n})\not=\emptyset\}$ is a
$p$-network at point $x$.} and improve a result of C. Liu.

We call a space $X$ is {\it Ohio complete} \cite{A} if in each
compactification $bX$ of $X$ there is a $G_{\delta}$-subset $Z$ such
that $X\subset Z$ and each point $y\in Z\setminus X$ is separated
from $X$ by a $G_{\delta}$-subset of $Z$.

\begin{lemma}\label{15}
Let $X$ be a $p$-space and every compact subset of $bX\setminus X$
be metrizalbe. Then there exists a $G_{\delta}$-subset $Y$ of $bX$
such that $X\subset Y$ and satisfies the following conditions:

\begin{enumerate}
\item $bX$ is first countable at every point $y\in Y\setminus X$;

\item If $X$ is a topological group and $\overline{Y\setminus X}\cap X\neq\emptyset$, then $X$ is
metrizable.
\end{enumerate}
\end{lemma}

\begin{proof}
Since $X$ is a $p$-space, $X$ is Ohio complete \cite[Corollary
3.7]{A}. It follows that there is a $G_{\delta}$-subset $Y$ of $bX$
such that $X\subset Y$ and every point $y\in Y\setminus X$ can be
separated from $X$ by a $G_{\delta}$-subset. We now prove that $Y$
satisfies the conditions (1) and (2).

(1) From the choice of $Y$, it is easy to see that for every point
$y\in Y\setminus X$ there exists a compact $G_{\delta}$-subset $C$
of $bX$ such that $y\in C\subset Y\setminus X\subset bX\setminus X$.
Since $C$ is compact, the compact subset $C$ is metrizable. Therefore, $y$ is a
$G_{\delta}$-point in $bX$ and hence, $bX$ is first countable at
$y$.

(2) We choose a point $a\in \overline{Y\setminus X}\cap X$. Since
$X$ is a $p$-space, there exists a compact subset $F$ of $X$ such
that $a\in F$ and $F$ has a countable base of neighborhoods in $X$. Since $X$ is dense in $bX$, the set $F$ has a countable base of open
neighborhoods $\phi =\{U_{n}: n\in \omega\}$ in $bX$. Since $a\in
\overline{Y\setminus X}$, we can fix a $b_{n}\in U_{n}\cap
(Y\setminus X)$ for each $n\in \omega$. Obviously, there is a point
$c\in F$ which is a limit point for the sequence $\{b_{n}\}$. By
(1), we know that $bX$ is first countable at $b_{n}$ for every $n\in
\omega$. We can fix a countable base $\eta_{n}$ of $bX$ at $b_{n}$.
Then $\cup\{\eta_{n}:n\in \omega\}$ is a countable $\pi$-base of
$bX$ at $c$. Then the space $X$ also has a countable $\pi$-base at $c$,
since $c\in X$ and $X$ is dense in $bX$. Since $X$ is a topological
group, the space $X$ is metrizable.
\end{proof}

\begin{theorem}\label{t2}
Let $G$ be a non-locally compact topological group. If
$Y=bG\setminus G$ has a quasi-$G_{\delta}$-diagonal. Then $G$ and
$bG$ are separable and metrizable.
\end{theorem}

\begin{proof}
Obviously, $Y$ has a countable pseudocharacter. By \cite[Theorem
5.1]{A2}, $G$ is a paracompact $p$-space or $Y$ is first countable.

Case 1: The space $Y$ is first countable.

From \cite[Proposition 2.3]{BH} it follows that each countably
compact subset of $Y$ is a compact, metrizable, $G_{\delta}$-subset
of $Y$. Note that a Lindel\"{o}f $p$-space with a
quasi-$G_{\delta}$-diagonal is metrizable by \cite[Corollary
3.6]{HR}. In view of the proof of Theorem~\ref{t0}, it is easy to
see that $G$ and $bG$ are separable and metrizable.

Case 2: The space $G$ is a paracompact $p$-space.

By \cite[Corollary 3.7]{A}, $G$ is Ohio complete. Therefore, there
exists a $G_{\delta}$-subset $X$ of $bG$ such that $G\subset X$ and
every point $x\in X\setminus G$ can be separated from $G$ by a
$G_{\delta}$-set of $X$. Let $M=X\setminus G$. Then $bG$ is first
countable at every point $y\in M$ by Lemma~\ref{15}.

Subcase 1: $\overline{M}\cap G=\emptyset$. Then $X\setminus
\overline{M}=G$. Hence $G$ is a $G_{\delta}$-subset of $bG$. It
follows that $Y$ is $\sigma$-compact. Since $Y$ has a
quasi-$G_{\delta}$-diagonal, every compact subspace of $Y$ is
separable and metrizable by \cite[Proposition 2.3]{BH}. Hence $Y$ is
separable. Since both $Y$ and $G$ are dense in $bG$, it follows that
the souslin number of $G$ is countable. The space $G$ is Lindel\"{o}f, since
$G$ is paracompact. Therefore, $G$ is a Lindel\"{o}f $p$-space. Then
$Y$ is a Lindel\"{o}f $p$-space by Lemma~\ref{l0}. Since $Y$ has a
quasi-$G_{\delta}$-diagonal, the space $Y$ is metrizable by \cite[Corollary
3.6]{HR}. It follows that $Y$ has a $G_{\delta}$-diagonal.
Therefore, $G$ and $bG$ are separable and metrizable by
\cite[Theorem 5]{A1}.

Subcase 2: $\overline{M}\cap G\neq\emptyset$. Then $G$ is metrizable
by Lemma~\ref{15}.

Subcase 2(a): $G$ is locally separable. By \cite[Proposition
2.3]{BH}, it is easy to see that $G$ and $bG$ are separable and
metrizable by the proof of Theorem ~\ref{t0}.

Subcase 2(b): $G$ is nowhere locally separable. Fix a base
$\mathcal{B}=\cup\{\mathcal{U}_{n}:n\in \mathbb{N}\}$ of $G$ such
that each $\mathcal{U}_{n}$ is discrete in $G$. Let $F_{n}$ be the
set of all accumulation points for $\mathcal{U}_{n}$ in $bG$ for
each $n\in \mathbb{N}$. Put $Z=\cup\{F_{n}:n\in \mathbb{N}\}$. Then
$Z$ is dense in $Y$ and $\sigma$-compact by \cite[Proposition
4]{A1}. Since every compact space with a quasi-$G_{\delta}$-diagonal
is separable and metrizable by \cite[Proposition 2.3]{BH}, the space $Z$ has a
countable network. Because $G$ is nowhere locally compact, the space $Y$ is
dense in $bG$. It follows that $Z$ is dense in $bG$. Hence $bG$ is
separable, which implies that the Souslin number of $G$ is
countable. Since $G$ is metrizable, the space $G$ is separable. Then $Y$ is a
Lindel\"{o}f $p$-space by Lemma~\ref{l0}. Hence $Y$ is metrizable by
\cite[Corollary 3.6]{HR}. It follows that $Y$ is separable and
metrizable, which implies that $G$ and $bG$ are separable and
metrizable.
\end{proof}

\begin{lemma}\label{l3}
Let $X$ be a Lindel\"{o}f space with locally a
quasi-$G_{\delta}$-diagonal. Then $X$ has a
quasi-$G_{\delta}$-diagonal.
\end{lemma}

\begin{proof}
For each point $x\in X$, there exists an open neighborhood $U(x)$
such that $x\in U(x)$ and $U(x)$ has a quasi-$G_{\delta}$-diagonal.
Then $\mathcal{U}=\{U(x):x\in X\}$ is an open cover of $X$. Since
$X$ is a Lindel\"{o}f space, there exists a countable subfamily
$\mathcal{V}\subset\mathcal{U}$ such that $X=\cup\mathcal{V}$.
Denoted $\mathcal{V}$ by $\{U_{n}:n\in \mathbb{N}\}$. For each $n\in
\mathbb{N}$, let $\{\mathcal{U}_{nk}\}_{k\in \mathbb{N}}$ be a
quasi-$G_{\delta}$-diagonal sequence of $U_{n}$. Let
$\mathcal{F}=\{\mathcal{U}_{nk}\}_{n, k\in \mathbb{N}}$. Then
$\mathcal{F}$ is a quasi-$G_{\delta}$-diagonal sequence of $X$.

Indeed, for distinct points $x, y\in X$, there exists an $n\in
\mathbb{N}$ such that $x\in U_{n}$.

If $y\not\in U_{n}$, then $x\in U_{n}\subset X-\{y\}$. Since
$\{\mathcal{U}_{nk}\}_{k\in \mathbb{N}}$ is a
quasi-$G_{\delta}$-diagonal sequence of $U_{n}$, there exists a
$k\in \mathbb{N}$ such that $x\in \cup\mathcal{U}_{nk}$. Hence
$x\in\mbox{st}(x,
\mathcal{U}_{nk})\subset\cup\mathcal{U}_{nk}\subset U_{n}\subset
X-\{y\}$.

If $y\in U_{n}$, then $x\in U_{n}-\{y\}\subset X-\{y\}$. Since
$\{\mathcal{U}_{nk}\}_{k\in \mathbb{N}}$ is a
quasi-$G_{\delta}$-diagonal sequence of $U_{n}$, there exists a
$k\in \mathbb{N}$ such that $x\in \mbox{st}(x,
\mathcal{U}_{nk})\subset U_{n}-\{y\}\subset X-\{y\}$.

Therefore, $\mathcal{F}$ is a quasi-$G_{\delta}$-diagonal sequence
of $X$.
\end{proof}

\begin{theorem}\label{t3}
Let $G$ be a non-locally compact topological group. If
$Y=bG\setminus G$ has locally a quasi-$G_{\delta}$-diagonal, then
$G$ and $bG$ are separable and metrizable.
\end{theorem}

\begin{proof}
By \cite[Proposition 2.1 and 2.5]{BH} and Lemma~\ref{l6}, it is easy
to see that $G$ is locally a separable and metrizable space. Then
$Y$ is a Lindel\"{o}f space by Henriksen and Isbell's theorem. From
Lemma ~\ref{l3} it follows that $Y$ has a
quasi-$G_{\delta}$-diagonal. Then $G$ and $bG$ are separable and
metrizable by Theorem ~\ref{t2}.
\end{proof}

\begin{question}\label{q0}
Is there a topological group $G$ such that the $Y=bG\setminus G$ has a
$W_{\delta}$-diagonal\footnote{A space $X$ is said to have a {\it $W_{\delta}$-diagonal} if there is a sequence ($\mathcal{B}_{n}$) of bases for $X$ such that whenever
$x\in B_{n}\in\mathcal{B}_{n}$, and $(B_{n})$ is decreasing (by set inclusion), then $\{x\}=\cap\{B_{n}: n\in\omega\}$.}, $G$ is not reparable and metrizable?
\end{question}

\begin{corollary}\cite{LC}
Let $G$ be a non-locally compact topological group. If
$Y=bG\setminus G$ has locally a $G_{\delta}$-diagonal, then $G$ and
$bG$ are separable and metrizable.
\end{corollary}

Next, we study the remainder that are the unions of the
$G_{\delta}$-diagonals.

\begin{lemma}\label{17}
Let $G$ be a non-locally compact topological group. If there exists
a point $a\in Y=bG\setminus G$ such that $\{a\}$ is a
$G_{\delta}$-set in $Y$, then at least one of the following
conditions holds:
\begin{enumerate}
\item $G$ is a paracompact $p$-space;

\item $Y$ is first-countable at some point.
\end{enumerate}
\end{lemma}

\begin{proof}
Suppose that $Y$ is not first-countable at point $a$. Since $a$ is a
$G_{\delta}$-point in $Y$, there exists a compact subset $F\subset
bG$ with a countable base at $F$ in $bG$ such that
$\{a\}=F\cap (bG\setminus G)$. We have
$F\setminus\{a\}\neq\emptyset$, since $Y$ is not first-countable at
point $a$. Therefore, there exists a non-empty compact subset
$B\subset F$ with a countable base at $B$ in $bG$.
Obviously, $B\subset G$. It follows that $G$ is a topological group
of countable type \cite{RW}. Therefore, $G$ is a paracompact
$p$-space\cite{RW}.
\end{proof}

\begin{lemma}\label{18}
Let $G$ be a non-locally compact topological group, and
$Y=bG\setminus G=Y_{1}\cup Y_{2}$, where both $Y_{1}$ and $Y_{2}$
have a countable pseudocharacter. If at most one of the $Y_{1}$ and
$Y_{2}$ is dense in $bG$, then at least one of the following
conditions holds:
\begin{enumerate}
\item $G$ is a paracompact $p$-space;

\item $Y$ is first-countable at some point.
\end{enumerate}
\end{lemma}

\begin{proof}
Without loss of generality, we can assume that $\overline{Y_{1}}\neq
bG$. Let $U=bG\setminus\overline{Y_{1}}$. Then $V=U\cap Y=U\cap
Y_{2}\neq\emptyset$. It follows that $V$ is an open subset of $Y$
and each point of $V$ is a $G_{\delta}$-point. By Lemma ~\ref{17},
we complete the proof.
\end{proof}

\begin{theorem}\label{t5}
Let $G$ be a non-locally compact topological group, and
$Y=bG\setminus G=Y_{1}\cup Y_{2}$, where both $Y_{1}$ and $Y_{2}$
have a countable pseudocharacter. If both $Y_{1}$ and $Y_{2}$ are
Ohio complete, then at least one of the following conditions holds:
\begin{enumerate}
\item $G$ is a paracompact $p$-space;

\item $Y$ is first-countable at some point.
\end{enumerate}
\end{theorem}

\begin{proof}
Case 1: $\overline{Y_{1}}\neq bG$ or $\overline{Y_{2}}\neq bG$.

It is easy to see by Lemma ~\ref{18}.

Case 2: $\overline{Y_{1}}=bG$ and $\overline{Y_{2}}=bG$.

Then $bG$ is the Hausdorff compactification of $Y_{1}$ and $Y_{2}$.
Since $Y_{1}$ and $Y_{2}$ are Ohio complete, there exist
$G_{\delta}$-subsets $X_{1}$ and $X_{2}$ satisfy the definition of
Ohio complete, respectively.

Case 2(a): $Y_{1}=X_{1}$ and $Y_{2}=X_{2}$.

Then $Y$ has countable pseudocharacter. By \cite[Theorem 5.1]{A2},
we complete the proof.

Case 2(b): $Y_{1}\neq X_{1}$ or $Y_{2}\neq X_{2}$.

Without loss of generality, we can assume that $Y_{1}\neq X_{1}$. If
$(X_{1}\setminus Y_{1})\cap Y_{2}\neq\emptyset$, then for each $y\in
(X_{1}\setminus Y_{1})\cap Y_{2}$ there exists a compact subset $C$
such that $y\in C$ and $C\cap Y_{1}=\emptyset$. Obviously, $y$ is a
$G_{\delta}$-point of $Y$. By Lemma ~\ref{17}, we also complete the
proof. If $(X_{1}\setminus Y_{1})\cap Y_{2}=\emptyset$, then there
exists a compact subset $C\subset G$ with a countable base at $C$ in $bG$. It follows that $G$ is a topological group of
countable type \cite{RW}. Therefore, $G$ is a paracompact $p$-space
\cite{RW}.
\end{proof}

A space with a $G_{\delta}$-diagonal is Ohio complete\cite{A4}.
Therefore, by Theorem ~\ref{t5}, we have the following result.

\begin{theorem}
Let $G$ be a non-locally compact topological group, and
$Y=bG\setminus G=Y_{1}\cup Y_{2}$, where both $Y_{1}$ and $Y_{2}$
have a $G_{\delta}$-diagonal. Then at least one of the following
conditions holds:
\begin{enumerate}
\item $G$ is a paracompact $p$-space;

\item $Y$ is first-countable at some point.
\end{enumerate}
\end{theorem}

\begin{question}
Let $G$ be a non-locally compact topological group, and
$Y=bG\setminus G=\bigcup_{i=1}^{i=n}Y_{i}$, where $Y_{i}$ has a
$G_{\delta}$-diagonal for every $1\leq i\leq n$. Is $G$  a
paracompact $p$-space or is $Y$ first-countable at some point?
\end{question}

\begin{question}
Let $G$ be a non-locally compact topological group, and
$Y=bG\setminus G=Y_{1}\cup Y_{2}$, where both $Y_{1}$ and $Y_{2}$
have quasi-$G_{\delta}$-diagonal. Is $G$  a paracompact $p$-space or
is $Y$ first-countable at some point?
\end{question}

\medskip
\section{Remainders of locally BCO and locally hereditarily D-spaces}

First, we study the following question, which was posed by C. Liu in
\cite{LC}.

\begin{question}
Let $G$ be a non-locally compact topological group, and
$Y=bG\setminus G$ have a BCO\footnote{A space $X$ is said to have a
{\it base of countable order}(BCO) \cite{Gr} if there is a sequence
$\{\mathcal {B}_{n}\}$ of base for $X$ such that whenever $x\in
b_{n}\in\mathcal {B}_{n}$ and $(b_{n})$ is decreasing (by set
inclusion), then $\{b_{n}: n\in \mathbb{N}\}$ is a base at $x$.}.
Are $G$ and $bG$ separable and metrizable?
\end{question}

\medskip
Now we give a partial answer for Question 4.1.

\begin{theorem}\label{t4}
Let $G$ be a non-locally compact topological group, and
$Y=bG\setminus G$ has a BCO. If $Y$ is Ohio complete, then
$G$ and $bG$ are separable and metrizable.
\end{theorem}

\begin{proof}
Since $Y$ is Ohio
complete, $G$ is a paracompact $p$-space or $\sigma$-compact space
by \cite[Theorem 4.3]{A}.

Case 1: The space $G$ is a paracompact $p$-space.

Since $G$ is a $p$-space, the space $Y$ is Lindel\"{o}f by Henriksen and
Isbell's theorem. Hence $Y$ is developable by \cite[Theorem
6.6]{Gr}. Then $G$ and $bG$ are separable and metrizable by Theorem
~\ref{t3}.

Case 2: The space $G$ is a $\sigma$-compact space.

We claim that $G$ is metrizable. Suppose that $G$ is not metrizable.
Then $Y$ is $\omega$-bounded\footnote{A space $X$ is said to be {\it
$\omega$-bounded} if the clourse of every countable subset of $X$ is
compact.} by \cite[Theorem 3.12]{A2}. Since $G$ is a
$\sigma$-compact topological group, the Souslin number $c(G)$ of $G$
is countable by a theorem of Tkachenko \cite[Corollary 2]{TG}.
Therefore, $c(bG)\leq\omega$. $Y$ is dense in $bG$, since $G$ is
non-locally compact. It follows that $c(Y)\leq\omega$ as well. Since
$Y$ is $\check{C}$ech-complete, there exists a dense subspace
$Z\subset Y$ such that $Z$ is a paracompact and
$\check{C}$ech-complete subspace of $Y$ by \cite{SB}. Then $Z$ is a
paracompact space with a BCO. Therefore, $Z$ is metrizable by
\cite[Theorem 1.2 and 6.6]{Gr}. Since $c(Y)\leq\omega$ and $Z$ is
dense for $Y$, $c(Z)\leq\omega$ as well. It follows that $Z$ is
separable. Since $Y$ is $\omega$-bounded, it is compact. Therefore,
$G$ is locally compact, which is a contradiction. It follows that $G$ is
metrizable. Therefore, $G$ and $bG$ are separable and metrizable by
Case 1.
\end{proof}

\begin{theorem}
Let $G$ be a non-locally compact topological group, and
$Y=bG\setminus G$ have a BCO. If  $G$ is an $\Sigma$-space,
then $G$ and $bG$ are separable and metrizable.
\end{theorem}

\begin{proof}
From \cite[Theorem 2.8]{AB} it follows that every compact subspace of $Y$
has countable character in $Y$. Since $G$ is non-locally compact,
$Y$ is also a dense subset of $bG$. Hence $G$ is Lindel\"{o}f space
by Henriksen and Isbell's theorem. If $G$ is a $\sigma$-compact
space, then $G$ and $bG$ are separable and metrizable by Case 2 in
Theorem ~\ref{t4}. Hence we assume that $G$ is non-$\sigma$-compact.
Since $G$ is a Lindel\"{o}f $\Sigma$-space, it is easy to see that
$G$ is a Lindel\"{o}f $p$-space by the proof of \cite[Theorem
4.2]{A2}. It follows that $G$ and $bG$ are separable and metrizable
by Case 1 in Theorem ~\ref{t4}.
\end{proof}

Finally, we study the remainders of topological groups with locally a hereditarily D-space.

\begin{theorem}\label{t6}
Let $G$ be a topological group. If for each $y\in Y=bG\setminus G$
there exists an open neighborhood $U(y)$ of $y$ such that every
$\omega$-bounded subset of $U(y)$ is compact, then at least one of
the following conditions holds:
\begin{enumerate}
\item $G$ is metrizable;

\item $bG$ can be continuously mapped onto the Tychonoff cube $I^{\omega_{1}}$.
\end{enumerate}
\end{theorem}

\begin{proof}
Case 1: The space $G$ is locally compact.

If $G$ is not metrizable, then $G$ contains a topological copy of
$D^{\omega_{1}}$. Since the space $G$ is normal, the space $G$ can be
continuously mapped onto the Tychonoff cube $I^{\omega_{1}}$

Case 2: The space G is not locally compact.

Obviously, both $G$ and $Y$ are dense in $bG$. Suppose that the
condition (2) doesn't hold. Then, by a theorem of
$\check{S}$apirovski\v{\i} in \cite{SB1}, the set $A$ of all points
$x\in bG$ such that the $\pi$-character of $bG$ at $x$ is countable
is dense in $bG$. Since $G$ is dense in $bG$, it can follow that the
$\pi$-character of $G$ is countable at each point of $A\cap G$.

Subcase 2(a): $A\cap G\neq\emptyset$.

Since $G$ is a topological group, it follows that $G$ is first
countable, which implies that $G$ is metrizable.

Subcase 2(b): $A\cap G=\emptyset$.

Obviously, $A\subset Y$. For each $y\in Y$, there exists an open
neighborhood $U(y)$ in $Y$ such that $y\in U(y)$ and every
$\omega$-bounded subset of $U(y)$ is compact. Obviously,  $A\cap
U(y)$ is dense of $U(y)$. Also, it is easy to see that $A\cap U(y)$
is $\omega$-bounded subset for $U(y)$. Therefore, $A\cap U(y)$ is
compact. Then $A\cap U(y)=U(y)$, since $A\cap U(y)$ is dense of
$U(y)$. Hence $Y$ is locally compact, a contradiction.
\end{proof}

A {\it neighborhood assignment} for a space $X$ is a function
$\varphi$ from $X$ to the topology of $X$ such that $x\in \varphi
(x)$ for each point $x\in X$. A space $X$ is a {\it
D-space} \cite{DV}, if for any neighborhood assignment $\varphi$ for
$X$ there is a closed discrete subset $D$ of $X$ such that
$X=\bigcup_{d\in D}\varphi (d)$.

It is easy to see that every countably compact D-space is compact.
Hence we have the following result by Theroem~\ref{t6}.

\begin{theorem}\label{t7}
Let $G$ be a topological group. If $Y=bG\setminus G$ is locally a
hereditarily D-space, then at least one of the following conditions
holds:
\begin{enumerate}
\item $G$ is metrizable;

\item $bG$ can be continuously mapped onto the Tychonoff cube $I^{\omega_{1}}$.
\end{enumerate}
\end{theorem}

{\bf Acknowledgements}. We wish to thank
the reviewers for the detailed list of corrections, suggestions to the paper, and all her/his efforts
in order to improve the paper. In particular, Question~\ref{q0} is due to the reviewers.

\vskip0.9cm

\end{document}